\documentclass[preprint,12pt]{elsarticle}

\usepackage{amssymb}

\usepackage{amsthm}

\journal{Journal of Pure and Applied Algebra}
\usepackage{amsmath,amssymb,url}

\numberwithin{equation}{section}

\theoremstyle{plain}
\newtheorem{theorem}{Theorem}[section]
\newtheorem{lemma}[theorem]{Lemma}
\newtheorem{proposition}[theorem]{Proposition}
\newtheorem{corollary}[theorem]{Corollary}

\theoremstyle{definition}
\newtheorem{definition}[theorem]{Definition}

\newtheorem{remark}[theorem]{Remark}

\begin{document}

\begin{frontmatter}

\ead{paul.lescot@univ-rouen.fr}
\ead[url]{http://www.univ-rouen.fr/LMRS/Persopage/Lescot/}

\title{Absolute algebra \\
III--The saturated spectrum}

\author{Paul Lescot}

\address{LMRS \\
CNRS UMR 6085 \\
UFR des Sciences et Techniques \\
Universit\'e de Rouen \\
Avenue de l'Universit\'e BP12 \\
76801 Saint-Etienne du Rouvray (FRANCE) }

\begin{abstract}
We investigate the algebraic and topological preliminaries to a geometry in characteristic $1$.
\end{abstract}


\end{frontmatter}

\newpage

\section{Introduction}
\label{S:intro}

The theory of \it characteristic $1$ semirings \rm (\it i.e. \rm semirings with $1+1=1$)
originated in many different contexts : pure algebra (see \it e.g. \rm LaGrassa's PhD thesis \cite{8}),
idempotent analysis and the study of $\mathbf R_{+}^{max}$
(\cite{1,3}), and Zhu's theory (\cite{12}), itself inspired by considerations of Hopf algebras
(see \cite{11}). Its main motivation is now the Riemann Hypothesis, via adeles and the theory of hyperrings (cf. \cite{2,3,4},
notably \S 6 from \cite{4}). 

For example, it has by now become clear (see \cite{4},Theorem 3.11)
that the classification of finite hyperfield extensions of the Krasner hyperring $K$ is one of the main problems of the theory.
If $H$ denotes an hyperring extension of $K$, $B_{1}$ the smallest characteristic one
semifield and $S$ the sign hyperring, then there
are canonical mappings $B_{1}\rightarrow S \rightarrow K \rightarrow H$,
whence mappings
$$
Spec(H) \rightarrow Spec(K) \rightarrow Spec(S) \rightarrow Spec(B_{1})\,\, ,
$$
thus $Spec(H)$ \lq\lq lies over\rq\rq $Spec(B_{1})$
(see \cite{4}, \S 6, notably diagram $(43)$, where $B_{1}$ is denoted by $\mathbf B$).

The ultimate goal of our investigations is to provide a proper algebraic geometry 
in characteristic one. The natural procedure is to construct \lq\lq affine $B_{1}$--schemes\rq\rq and endow them
with an appropriate topology and a sheaf of semirings ; a suitable glueing procedure will then produce general
\lq\lq $B_{1}$--schemes\rq\rq . This program is not yet completed ; in this paper, we deal 
with a natural first step : the extension to  $B_{1}$--algebras of the notions of spectrum and Zariski topology,
and the fundamental topological properties of these objects. In order to construct a structure sheaf over the spectrum of a $B_{1}$--algebra, 
Castella's localization procedure (\cite{1}) will probably be useful.

As in our two previous papers, we work in the context of $B_{1}$--algebras, \it i.e. \rm
characteristic one semirings. For such an $A$, one may define prime ideals by analogy
to classical commutative algebra. 
In order to define the spectrum of a $B_{1}$--algebra $A$, two candidates readily suggest themselves :
the set $Spec(A)$ of prime (in a suitable sense) congruences, and the set $Pr(A)$ of prime ideals ;
in contrast to the classical situation, these two approaches are not equivalent.
In fact both sets may be equipped with a natural topology of Zariski type (see \cite{10}, Theorem 2.4 and Proposition 3.15),
but they do not in general correspond bijectively to one another ; nevertheless, the subset $Pr_{s}(A)\subseteq Pr(A)$ of \it saturated \rm prime ideals is in
natural bijection with the set of \it excellent \rm prime congruences (see below) on $A$.

It turns out (\S 3) that there is
another, far less obvious, bijection between $Pr_{s}(A)$ and the \it maximal spectrum \rm 
$MaxSpec(A)\subseteq Spec(A)$ of $A$. This mapping is actually an homeomorphism for the natural
(Zariski--type) topologies mentioned above. As a by--product, we find a new point of view on the 
descrption of the maximal spectrum
of the polynomial algebra $B_{1}[x_{1},...,x_{n}]$ found in \cite{9} and \cite{12}.
The homeomorphism in question is actually functorial in $A$ (\S 4).

In \S 5, we show that the theory of the nilradical and of the root of an ideal carry over,
with some precautions, to our setting ; the situation is even better when
one restricts oneself to \it saturated \rm ideals. This allows us, in \S 6, to establish some nice topological properties of
$$
MaxSpec(A)\simeq Pr_{s}(A) \,\, ;
$$
namely, it is $T_{0}$ and quasi--compact (Theorem 6.1), and the open quasi--compact sets constitute a basis stable under finite intersections.
Furthermore this space is \it sober \rm , \it i.e. \rm each irreducible closed set has a (necessarily unique) generic point. In other words, $Pr_{s}(A)$
satisfies the usual properties of a ring spectrum that are used in algebraic geometry
(see \it e.g. \rm the canonical reference \cite{6}):
$Pr_{s}(A)$ is a \it spectral space \rm in the sense of Hochster(\cite{7}).

In the last paragraph, we discuss the particular case of a \it monogenic \rm
$B_{1}$--algebra, that is, a quotient of the polynomial algebra
$B_{1}[x]$ ; in \cite{9}, we had listed the smallest finite such algebras.

In a subsequent work I shall investigate how higher concepts and methods of commutative algebra
(minimal prime ideals, zero divisors, primary decomposition) carry over to 
characteristic one semirings.

\newpage

\section{Definitions and notation}
\label{sec:2}

We shall review some the definitions and notation of our previous two papers 
(\cite{9}, \cite{10}).

$B_{1}=\{0,1\}$ denotes the smallest \it characteristic one semifield \rm ;
the operations of addition and multiplication are the obvious ones, with the 
slight change that
$$
1+1=1 \,\, .
$$

A $B_{1}$--module $M$ is a nonempty set equipped with an action
$$
B_{1}\times M \rightarrow M
$$
satisfying the usual axioms (see \cite{9}, Definition 2.3); as first seen in \cite{12}, Proposition 1
(see also \cite{9}, Theorem 2.5), $B_{1}$--modules can be canonically identified 
with ordered sets having a smallest element ($0$) and in which any two elements $a$ and $b$ 
have a least upper bound ($a+b$).
In particular, one may identify finite $B_{1}$--modules and nonempty finite
lattices.

A (commutative) $B_{1}$--algebra is a $B_{1}$--module equipped with an associative multiplication that
has a neutral element and satisfies the usual axioms relative to addition
(see \cite{9}, Definition 4.1). In the sequel, except when otherwise indicated, $A$
will denote a $B_{1}$--algebra.

An \it ideal \rm $I$ of $A$ is by definition a subset containing $0$,
stable under addition, and having the property that
$$
\forall x\in A \,\,\, \forall y\in I \,\, xy\in I \,\, ;
$$
$I$ is termed prime if $I\neq A$ and
$$
ab\in I \implies a\in I \,\, \text{or} \,\, b\in I \,\, .
$$

By a \it congruence \rm on $A$, we mean an equivalence relation on $A$ compatible
with the operations of addition and multiplication. The trivial congruence $\mathcal C_{0}(A)$
is characterized by the fact that any two elements of $A$ are equivalent under it ; 
the congruences are naturally ordered by inclusion, and $$MaxSpec(A)$$ will denote
the set of maximal nontrivial congruences on $A$.

For $\mathcal R$ a congruence on $A$, we set
$$
I(\mathcal R):=\{x\in A \vert x \,\, \mathcal R \,\, 0\} \,\, ;
$$
it is an ideal of $A$

A nontrivial congruence $\mathcal R$ is termed \it prime \rm if
$$
ab \,\, \mathcal R \,\, 0 \implies a \,\, \mathcal R \,\, 0 \,\, \text{or} \,\, b \,\, \mathcal R \,\, 0 \,\, ;
$$
the set of prime congruences on $A$ is denoted by $Spec(A)$. It turns out that
(see \cite{10}, Proposition 2.3)
$$
MaxSpec(A)\subseteq Spec(A) \,\, .
$$

For $J$ an ideal of $A$, there is a unique smallest congruence $\mathcal R_{J}$ 
such that $J\subseteq I(\mathcal R)$ ; it is denoted by $\mathcal R_{J}$.
Such congruences are termed \it excellent \rm.

An ideal $J$ of $A$ is termed \it saturated \rm if it is of the form
$I(\mathcal R)$ for some congruence $\mathcal R$ ; this is the case if and only if
$J=\overline{J}$, where
$$
\overline{J}:=I(\mathcal R_{J})\,\, .
$$

We shall denote the set of prime ideals of $A$ by $Pr(A)$, and the set 
of saturated prime ideals by $Pr_{s}(A)$.

For $S\subseteq A$, let us set
$$
W(S):=\{\mathcal P \in Pr(A)\vert S \subseteq \mathcal P\}\,\, ,
$$
and
$$
V(S):=\{\mathcal R \in Spec(A)\vert S \subseteq I(\mathcal R)\}\,\, .
$$
As seen in \cite{10}, Theorem 2.4 and Proposition 3.4, the family $(W(S))_{S\subseteq A}$ is the family of closed sets for a topology on $Pr(A)$,
and the family $(V(S))_{S\subseteq A}$ is the family of closed sets for a topology on $Spec(A)$.
We shall always consider $Spec(A)$ and $Pr(A)$ as equipped with these topologies, and their subsets with the induced topologies.

For $M$ a commutative monoid, we define the \it Deitmar spectrum \rm $Spec_{\mathcal D}(M)$
as the set of prime ideals (including $\emptyset$) of $M$ (in \cite{5}, this is denoted by $Spec \,\, \mathbf F_{M}$).
We define $\mathcal F(M)=B_{1}[M]$ as the 
\lq\lq monoid algebra of $M$ over $B_{1}$ \rq\rq ;
the functor $\mathcal F$ is adjoint to the forgetful functor from the category of $B_{1}$--algebras to the
category of monoids (for the details, see \cite{9}, \S 5).
Furthermore, there is an explicit canonical bijection between $Spec_{\mathcal D}(M)$
and a certain subset of $Spec (\mathcal F(M))$ (see \cite{10}, Theorem 4.2).

For $S$ a subset of $A$,
let $<S>$ denote the intersection of all the ideals of $A$ containing $S$ (there is always
at least one such ideal : $A$ itself). It is clear that $<S>$ is an ideal of $A$, and therefore
is the smallest ideal of $A$ containing $S$. As in ring theory, one may see that
$$
<S>=\{\sum_{j=1}^{n}a_{j}s_{j}\vert n\in \mathbf N, (a_{1},...,a_{n})\in A^{n},(s_{1},...,s_{n})\in S^{n}\}.
$$

We shall denote by $\mathcal{SP}$ the category whose objects are spectra of $B_{1}$--algebras and whose morphisms
are the continuous maps between them.
\newpage
\section{A new description of maximal congruences}
\label{sec:3}

Let $A$ denote a $B_{1}$--algebra.

For $\mathcal P$ a saturated prime ideal of $A$, let us define a relation 
$\mathcal S_{\mathcal P}$ on $A$ by :
$$
x\mathcal S_{\mathcal P} y \equiv (x\in \mathcal P \,\, \text{and} \,\, y\in \mathcal P) \,\, 
\text{or}
 (x\notin \mathcal P \,\, \text{and} \,\, y\notin \mathcal P)\,\, .
$$
Then $\mathcal S_{\mathcal P}$ is a congruence on $A$ : if $x\mathcal S_{\mathcal P} y$ and 
$x^{'}\mathcal S_{\mathcal P} y{'}$, then one and only one of the following holds :
\begin{enumerate}
\item[(i)] $x\in \mathcal P$, $y\in \mathcal P$, $x^{'}\in \mathcal P$ and $y^{'}\in \mathcal P$ \,\, ,
\item[(ii)] $x\in \mathcal P$, $y\in \mathcal P$, $x^{'}\notin \mathcal P$ and $y^{'}\notin \mathcal P$ \,\, ,
\item[(iii)] $x\notin \mathcal P$, $y\notin \mathcal P$, $x^{'}\in \mathcal P$ and $y^{'}\in \mathcal P$ \,\, ,
\item[(iv)] $x\notin \mathcal P$, $y\notin \mathcal P$, $x^{'}\notin \mathcal P$ and $y^{'}\notin \mathcal P$ \,\, .
\end{enumerate}

In case $(i)$, $x+x^{'}\in \mathcal P$ and $y+y^{'}\in \mathcal P$, whence
$x+x^{'}\mathcal S_{\mathcal P}y+y^{'}$ ; in cases $(ii)$ and $(iv)$, 
$x+x^{'}\notin \mathcal P$ and $y+y^{'}\notin \mathcal P$ (as $\mathcal P$ is saturated), whence
$x+x^{'}\mathcal S_{\mathcal P}y+y^{'}$. Case $(iii)$ is symmetrical relatively to case $(ii)$, therefore, in all
cases, $x+x^{'}\mathcal S_{\mathcal P}y+y^{'}$ : $\mathcal S_{\mathcal P}$ is compatible with addition.

In cases $(i)$, $(ii)$ and $(iii)$, $xx^{'}\in \mathcal P$ and $yy^{'}\in \mathcal P$, whence 
$xx^{'}\mathcal S_{\mathcal P}yy^{'}$ ; in case $(iv)$ $xx^{'}\notin \mathcal P$ and $yy^{'}\notin \mathcal P$
(as $\mathcal P$ is prime), whence also $xx^{'}\mathcal S_{\mathcal P}yy^{'}$ : $\mathcal S_{\mathcal P}$
is compatible with multiplication, hence is a congruence on $A$.

As $0\in\mathcal P$ and $1\notin \mathcal P$, $0\not{\mathcal S_{\mathcal P}} 1$, therefore ${\mathcal S_{\mathcal P}}$
is nontrivial ; but each $x\in A$ is either in $\mathcal P$ (whence $x{\mathcal S_{\mathcal P}}0$) or not
(whence $x{\mathcal S_{\mathcal P}}1$). It follows that
$$
\displaystyle\frac{A}{\mathcal S_{\mathcal P}}=\{\bar{0},\bar{1}\}\simeq B_{1} \,\, ;
$$
in particular, $\mathcal S_{\mathcal P}$ is maximal : $\mathcal S_{\mathcal P}\in MaxSpec(A)$.

Obviously, $I(\mathcal S_{\mathcal P})=\mathcal P$. 

Furthermore, let $(x,y)\in A^{2}$
be such that $x\mathcal R_{\mathcal P} y$ ; then there is $z\in \mathcal P$ such that
$x+z=y+z$. If $x\in \mathcal P$ then $y+z=x+z\in \mathcal P$, whence $y\in \mathcal P$
(as $y+(y+z)=y+z$ and $\mathcal P$ is saturated) ; symmetrically, $y\in \mathcal P$ implies $x\in\mathcal P$, whence the assertions $(x\in\mathcal P)$ and $(y\in\mathcal P)$ are equivalent, and
$x\mathcal S_{\mathcal P} y$. We have shown that 
$$
\mathcal R_{\mathcal P}\leq \mathcal S_{\mathcal P} \,\, .
$$

We shall denote by $\alpha_{A}$ the mapping
\begin{eqnarray}
\alpha_{A}
&:& Pr_{s}(A)\rightarrow MaxSpec(A) \nonumber\\
&&\mathcal P\mapsto \mathcal S_{\mathcal P} \,\, .\nonumber
\end{eqnarray}

Let $\mathcal R\in MaxSpec(A)$ ; then $\mathcal R\in Spec(A)$, whence $I(\mathcal R)$ is prime ;
by Theorem 3.8 of \cite{10}, it is saturated,
\it i.e. \rm $I(\mathcal R)\in Pr_{s}(A)$. Let us set
$$
\beta_{A}(\mathcal R):=I(\mathcal R)\,\, .
$$
\begin{theorem}The mappings
$$
\alpha_{A}: Pr_{s}(A) \mapsto MaxSpec(A)
$$
and
$$
\beta_{A}: MaxSpec(A) \mapsto Pr_{s}(A)
$$
are bijections, inverse of one another. They are continuous for the topologies on $Pr_{s}(A)$ and $MaxSpec(A)$ induced
by the topologies on $Pr(A)$ and $Spec(A)$ mentioned above, whence $Pr_{s}(A)$ and $MaxSpec(A)$ are homeomorphic.
\end{theorem}
\begin{proof}
Let $\mathcal R\in MaxSpec(A)$ ; then
$$
\alpha_{A}(\beta_{A}(\mathcal R))=\alpha_{A}(I(\mathcal R))=\mathcal S_{I(\mathcal R)}\,\, .
$$

Let us assume $x\mathcal R y$ ; then, if $x\in I(\mathcal R)$ one
has $x \mathcal R 0$, whence $y \mathcal R 0$ and 
$y\in I(\mathcal R)$; by symmetry, $y\in I(\mathcal R)$ implies $x\in I(\mathcal R)$, thus
$(x\in I(\mathcal R))$ and $(y\in I(\mathcal R))$ are equivalent, \it i.e.  \rm $x\mathcal S_{I(\mathcal R)} y$.
We have proved that $\mathcal R \leq \mathcal S_{I(\mathcal R)}$. 
As $\mathcal R$ is maximal, we have $\mathcal R=\mathcal S_{I(\mathcal R)}$, whence 
$$
\alpha_{A}(\beta_{A}(\mathcal R))=\mathcal S_{I(\mathcal R)}=\mathcal R \,\, ,
$$
and
$$
\alpha_{A}\circ \beta_{A}=Id_{MaxSpec(A)} \,\, .
$$

Let now $\mathcal P \in Pr_{s}(A)$ ; then
\begin{eqnarray*}
 (\beta_{A}\circ \alpha_{A})(\mathcal P)
&=&\beta_{A}(\alpha_{A}(\mathcal P)) \\
&=&\beta_{A}(\mathcal S_{\mathcal P}) \\
&=&I(\mathcal S_{\mathcal P}) \\
&=&\mathcal P \,\, ,\\
\end{eqnarray*}
whence
$$
\beta_{A}\circ \alpha_{A}=Id_{Pr_{s}(A)} \,\, ,
$$
and the first statement follows.

Let now $F$  denote a closed subset of $Pr_{s}(A)$ ; then $F=G\cap Pr_{s}(A)$
for $G$ a closed subset of $Pr(A)$ and $G=W(S):=\{\mathcal P\in Pr(A)\vert S\subseteq \mathcal P\}$
for some subset $S$ of $A$. But then, for $\mathcal R\in MaxSpec(A)$,
$\mathcal R\in \beta_{A}^{-1}(F)$ if and only if $\beta_{A}(\mathcal R)\in F$, \it i.e. \rm
$I(\mathcal R)\in G \cap Pr_{s}(A)$, that is $I(\mathcal R)\in G$, or $S\subseteq I(\mathcal R)$,
which means $\mathcal R\in V(S)$. Thus
$$
\beta_{A}^{-1}(F)=V(S)\cap MaxSpec(A)
$$
is closed in $MaxSpec(A)$. We have shown the continuity of
$\beta_{A}$.

Let now $H\subseteq MaxSpec(A)$ be closed ; then $H=MaxSpec(A)\cap L$ for some closed subset
$L$ of $Spec(A)$, and $L=V(T)$ for some subset $T$ of $A$. Then a saturated prime ideal $\mathcal P$ of $A$ 
belongs to $\alpha_{A}^{-1}(H)$ if and only if $\alpha_{A}(\mathcal P)\in H$, that is
$$\mathcal S_{\mathcal P}\in MaxSpec(A) \cap L\,\, ,$$
 \it i.e. \rm $$\mathcal S_{\mathcal P}\in V(T)$$
or $T\subseteq I(\mathcal S_{\mathcal P})$. But $I(\mathcal S_{\mathcal P})=\mathcal P$ whence $\mathcal P$ 
belongs to $\alpha_{A}^{-1}(H)$ if and only if $T\subseteq \mathcal P$, that is
$$
\alpha_{A}^{-1}(H)=W(T)\cap  Pr_{s}(A) \,\, ,
$$
which is closed in $Pr_{s}(A)$.
\end{proof}

Let us consider the special case in which $A$ is in the image of $\mathcal F$ : $A=\mathcal F(M)$, for $M$ 
a commutative monoid.
Let $P$ be a prime ideal of $M$ ; as seen in \cite{10}, Theorem 4.2,  $\tilde{P}$ is a saturated prime ideal in $A$,
 and one obtains in this way a bijection
between $Spec_{\mathcal D}(M)$ and $Pr_{s}(A)$. The following is now obvious :
\begin{theorem}The mapping
\begin{eqnarray}\psi_{M}
&:& Spec_{\mathcal D}(M)\rightarrow MaxSpec(\mathcal F(M)) \nonumber \\
&& P\mapsto \alpha_{\mathcal F(M)}(\tilde{P}) \,\, \nonumber
\end{eqnarray}
is a bijection.
\end{theorem}
Two particular cases are of special interest :
\begin{enumerate}
\item $M$ is a group ; then $Spec_{\mathcal D}(M)=\{\emptyset\}$, whence $MaxSpec(\mathcal F(G))$ has exactly
one element.
\item $M=C_{n}:=<x_{1},...,x_{n}>$ is the free monoid on $n$ variables
$x_{1},...,x_{n}$.
Then the elements of $Spec_{\mathcal D}(M)$ are the $(P_{J})_{J\subseteq\{1,...,n\}}$, where
$$
P_{J}:=\bigcup_{j\in J}x_{j}C_{n}
$$
(a fact that was already used in \cite{10}, Example 4.3).
Then $$\psi_{M}(P_{J})=\alpha_{\mathcal F(M)}(\tilde{P_{J}})=\mathcal S_{\tilde{P_{J}}}$$
whence $x \psi_{M}(P_{J}) y$ if and only if either ($x\in \tilde{P_{J}}$ and $y\in \tilde{P_{J}}$)
or ($x\notin \tilde{P_{J}}$ and $y\notin \tilde{P_{J}}$). But we have seen in \cite{9}, Theorem 4.5,
that 
$$
\mathcal F(M)=B_{1}[x_{1},...,x_{n}]
$$ 
could be identified with the set of finite formal sums of elements of 
$M$.
Obviously, an element $x$ of $\mathcal F(M)$ belongs to $\tilde{P_{J}}$ if and only if at least one of its 
components involves at least one
factor $x_{j}(j\in J)$. It is now clear that, using the notation of \cite{9}, Definition 4.6
and Theorem 4.7, 

$$\psi_{M}(P_{J})=\widetilde{ }_{J}\,\, .$$ 
We hereby recover the description of 
$MaxSpec(B_{1}[x_{1},...,x_{n}])$ given
in \cite{9}(Theorems 4.7, 4.8 and 4.10).
\end{enumerate}

The following result will be useful

\begin{theorem}Any proper saturated ideal of a $B_{1}$--algebra $A$ is contained in a saturated prime ideal of $A$.
 \end{theorem}
\begin{proof} Let $J$ be a proper saturated ideal of $A$ ; as $I(\mathcal R_{J})=\overline{J}=J\neq A$,
 $\mathcal R_{J}\neq \mathcal C_{0}(A)$. By Zorn's Lemma, one has $\mathcal R_{J}\leq \mathcal R$
for some $\mathcal R\in MaxSpec(A)$. According to Theorem 2.1, $\mathcal R=\alpha_{A}(\mathcal P)=\mathcal S_{\mathcal P}$ for a 
saturated prime ideal $\mathcal P$ of $A$, therefore $\mathcal R_{\mathcal J}\leq \mathcal S_{\mathcal P}$
and
$$
J=\overline{J}=I(\mathcal R_{J})\subseteq I(\mathcal S_{\mathcal P})=\mathcal P\,\, .
$$
\end{proof}
\newpage
\section{Functorial properties of spectra}
\label{sec:4}

Let $\varphi : A \rightarrow C$ denote a morphism of $B_{1}$--algebras, and let $\mathcal R\in Spec(C)$.
We define a binary relation $\tilde{\varphi}(\mathcal R)$ on $A$ by :
$$
\forall (a,a^{'})\in A^{2} \,\, a \tilde{\varphi}(\mathcal R) a^{'}\equiv \varphi(a) \mathcal R \varphi(a^{'}) \,\, .
$$
It is clear that $\tilde{\varphi}(\mathcal R)$ is a congruence on $A$, and that
$$
I(\tilde{\varphi}(\mathcal R))=\varphi^{-1}(I(\mathcal R))\,\, .
$$
In particular $I(\tilde{\varphi}(\mathcal R))$ is a prime ideal of $A$,
hence $\tilde{\varphi}(\mathcal R)\in Spec(A)$ : $\tilde{\varphi}$ maps
$Spec(C)$ into $Spec(A)$.
Let $F:=V(S)$ be a closed subset of 
$Spec(A)$, and let $\mathcal R \in Spec(C)$ ; then $\mathcal R\in \tilde{\varphi}^{-1}(F)$
if and only if $\tilde{\varphi}(\mathcal R)\in F$, that is $S\subseteq I(\tilde{\varphi}(\mathcal R))$,
or $S\subseteq \varphi^{-1}(I(\mathcal R))$, \it i.e. \rm $\varphi(S)\subseteq I(\mathcal R)$,
or $\mathcal R\in V(\varphi(S))$. Therefore $\tilde{\varphi}^{-1}(F)=V(\varphi(S))$
is closed in $Spec(C)$ : $\tilde{\varphi}$ is continuous.

Furthermore, for $\varphi : A \rightarrow C$ and $\psi : C \rightarrow D$ one has
$$
\widetilde{\psi \circ \varphi}=\tilde{\varphi}\circ \tilde{\psi} : Spec(D)\rightarrow Spec(A)\,\, .
$$
It follows that the equations $\mathcal H(A)=Spec(A)$ and $\mathcal H(\varphi)=\tilde{\varphi}$
define a contravariant functor $\mathcal H$ from $\mathcal Z_{a}$ to $\mathcal{SP}$.

Let $J$ denote an ideal in $C$, and let us assume $a\mathcal R_{\varphi^{-1}(J)} a^{'}$ ; then there is
an $x\in \varphi^{-1}(J)$ with $a+x=a^{'}+x$.
Now $\varphi(x)\in J$ and
\begin{eqnarray}
\varphi(a)+\varphi(x)
&=&\varphi(a+x) \nonumber \\
&=&\varphi(a^{'}+x) \nonumber  \\
&=&\varphi(a^{'})+\varphi(x) \,\, ,\nonumber
\end{eqnarray}
whence $\varphi(a)\mathcal R_{J}\varphi(a^{'})$ and $a\tilde{\varphi}(\mathcal R_{J})a^{'}$.
We have established
\begin{proposition}
Let $A$ and $C$ denote $B_{1}$--algebras, $\varphi : A \rightarrow C$ a morphism
and $J$ an ideal of $C$ : then
$$
\mathcal R_{\varphi^{-1}(J)}\leq \tilde{\varphi}(\mathcal R_{J}) \,\, .
$$
\end{proposition}

\begin{theorem}
Let $A$ and $C$ denote two $B_{1}$--algebras, and $\varphi : A \rightarrow C$ a morphism.
Then $\tilde{\varphi}:Spec(C)\rightarrow Spec(A)$ maps $MaxSpec(C)$ into $MaxSpec(A)$,
and the diagram
$$
\begin{array}[c]{ccc}
Pr_{s}(C)&\stackrel{\varphi^{-1}}{\rightarrow}&Pr_{s}(A)\\
\downarrow\scriptstyle{\alpha_{C}}&&\downarrow\scriptstyle{\alpha_{A}}\\
MaxSpec(C)&\stackrel{\tilde{\varphi}}{\rightarrow}&MaxSpec(A)
\end{array}
$$
commutes.
\end{theorem}
\begin{proof}Let $\mathcal P\in Pr_{s}(C)$, then, for all $(a,a^{'})\in A^{2}$
\begin{eqnarray}
a \tilde{\varphi}(\mathcal S_{\mathcal P}) a^{'}
&\Longleftrightarrow&\varphi(a)\mathcal S_{\mathcal P}\varphi(a^{'}) \nonumber \\
&\Longleftrightarrow&(\varphi(a)\in \mathcal P \,\, \text{and} \,\,
\varphi(a^{'})\in \mathcal P)) \,\, \nonumber \\
\text{or} \,\, (\varphi(a)\notin \mathcal P \,\, \text{and} \,\,
\varphi(a^{'})\notin \mathcal P) \nonumber \\
&\Longleftrightarrow&(a\in \varphi^{-1}(\mathcal P) \,\, \text{and} \,\,
a^{'}\in \varphi^{-1}(\mathcal P)) \,\, \nonumber \\
\text{or} \,\,  (a\notin \varphi^{-1}(\mathcal P) \,\, \text{and} \,\,
a^{'}\notin \varphi^{-1}(\mathcal P))\nonumber \\
&\Longleftrightarrow&a\mathcal S_{\varphi^{-1}(\mathcal P)}a^{'} \,\, . \nonumber
\end{eqnarray}
Therefore 
\begin{eqnarray}
(\tilde{\varphi}\circ \alpha_{C})(\mathcal P)
&=&\tilde{\varphi}(\alpha_{C}(\mathcal P)) \nonumber \\
&=&\tilde{\varphi}(\mathcal S_{\mathcal P}) \nonumber \\
&=&\mathcal S_{\varphi^{-1}(\mathcal P)} \nonumber \\
&=&\alpha_{A}(\varphi^{-1}(\mathcal P)) \nonumber \\
&=&(\alpha_{A}\circ \varphi^{-1})(\mathcal P) \nonumber
\end{eqnarray}
whence $\tilde{\varphi}\circ \alpha_{C}=\alpha_{A}\circ \varphi^{-1} \,\, .$

Incidentally we have proved that $\tilde{\varphi}$ maps $MaxSpec(C)=\alpha_{C}(Pr_{s}(C))$
into $\alpha_{A}(Pr_{s}(A))=MaxSpec(A)$, \it i.e. \rm the first assertion.
\end{proof}
\newpage

\section{Nilpotent radicals and prime ideals}
\label{sec:5}

The usual theory generalizes without major problem to $B_{1}$--algebras.

\begin{theorem}In the $B_{1}$--algebra $A$,let us define
$$
Nil(A):=\{x\in A \vert (\exists n\geq 1) x^{n}=0\}\,\, .
$$
Then $Nil(A)$ is a saturated ideal of $A$, and one has
$$
\bigcap_{\mathcal P\in Pr(A)}\mathcal P=\bigcap_{\mathcal P\in Pr_{s}(A)}\mathcal P=Nil(A)\,\, .
$$
\end{theorem}
\begin{proof}
Let $M:=\bigcap_{\mathcal P\in Pr(A)}\mathcal P$ and $N=\bigcap_{\mathcal P\in Pr_{s}(A)}\mathcal P$. If 
$x\in Nil(A)$ and \linebreak $\mathcal P\in Pr(A)$,
then, for some $n\geq 1$, $x^{n}=0\in \mathcal P$, whence (as $\mathcal P$ is prime) 
$x\in \mathcal P$ : $Nil(A)\subseteq M$.

As $Pr_{s}(A)\subseteq Pr(A)$, we have $M\subseteq N$.

Let now $x\notin Nil(A)$ ; then
$$
(\forall n\in \mathbf N) \,\, x^{n}\neq 0 \,\, .
$$
Define
$$
\mathcal E:=\{J\in Id_{s}(A) \vert (\forall n\geq 0) x^{n}\notin J\}.
$$
This set is nonempty ($\{0\}\in \mathcal E$) and inductive for $\subseteq$,
therefore, by Zorn's Lemma, there exists a maximal element $\mathcal P$ of $\mathcal E$.
As $1=x^{0}\notin \mathcal P$, $\mathcal P \neq A$.

Let us assume $ab\in \mathcal P$, $a\notin \mathcal P$ and $b\notin \mathcal P$ ;
then $\overline{\mathcal P + Aa}$ and $\overline{\mathcal P + Ab}$ are saturated ideals of $A$
strictly containing $\mathcal P$, whence there exists two integers $m$ and $n$ with $x^{m}\in \overline{\mathcal P + Aa}$
and $x^{n}\in \overline{\mathcal P + Ab}$. By definition of the closure of an ideal, there \linebreak are  $u=p_{1}+\lambda a\in \mathcal P + Aa$
and $v=p_{2}+\mu b\in \mathcal P + Ab$ such that $x^{m}+u=u$ and $x^{n}+v=v$.
Then
$$
ub=p_{1}b+\lambda(ab)\in \mathcal P
$$
and
$$
x^{m}b+ub=(x^{m}+u)b=ub\,\, ,
$$
whence, as $\mathcal P$ is saturated, $x^{m}b\in \mathcal P$.

Then
$$
x^{m}v=x^{m}p_{2}+\mu x^{m}b\in \mathcal P \,\, ;
$$

as
\begin{eqnarray}
x^{m+n}+x^{m}v
&=&x^{m}(x^{n}+v) \nonumber \\
&=&x^{m}v \,\, , \nonumber
\end{eqnarray}
we obtain $x^{m+n}\in\mathcal P$, a contradiction.

Therefore $\mathcal P$ is prime and saturated and $x=x^{1}\notin \mathcal P$, whence $x\notin N$. We have
proved that $N\subseteq Nil(A)$, whence $M=N=Nil(A)$.
\end{proof}
\begin{corollary}
$$Nil(A)=\bigcap_{\mathcal P\in Pr(A)}\overline{\mathcal P}\,\, .$$
\end{corollary}
\begin{proof}
\begin{eqnarray}
Nil(A)
&=&\bigcap_{\mathcal P\in Pr(A)}\mathcal P \,\, \text{(by Theorem 5.1)} \nonumber \\
&\subseteq&\bigcap_{\mathcal P\in Pr(A)}\overline{\mathcal P} \nonumber \\
&\subseteq&\bigcap_{\mathcal P\in Pr_{s}(A)}\overline{\mathcal P} \nonumber \\
&=&\bigcap_{\mathcal P\in Pr_{s}(A)}\mathcal P \nonumber \\
&=& Nil(A) \,\, (\text{also by Theorem 5.1}). \nonumber 
\end{eqnarray}
\end{proof}

\begin{definition}For $I$ an ideal of $A$, we define the \it root \rm $r(I)$
of $I$ by
$$
r(I):=\{x\in A \vert (\exists n\geq 1) x^{n}\in I\}.
$$
\end{definition}
\begin{lemma}
\begin{enumerate}

\item[(i)]$r(I)$ is an ideal of $A$.

\item[(ii)]$\overline{r(I)}\subseteq r(\overline{I})$ ; in particular,
if $I$ is saturated then so is $r(I)$.

\item[(iii)]$r(\{0\})=Nil(A)$.

\end{enumerate}
\end{lemma}
\begin{proof}
\begin{enumerate}
\item[(i)]Obviously, $0\in r(I)$.

If $x\in r(I)$ and $y\in r(I)$, then $x^{m}\in I$
for some $m\geq 1$ and $y^{n}\in I$ for some $n\geq 1$,
whence
\begin{eqnarray}
(x+y)^{m+n-1}
&=&\sum_{j=0}^{m+n-1}\binom {m+n-1}j x^{j}y^{m+n-1-j} \nonumber \\
(&&=\sum_{j=0}^{m+n-1} x^{j}y^{m+n-1-j}) \nonumber  \\
&&\in I \,\, ,\nonumber
\end{eqnarray}
as $x^{j}\in I$ for $j\geq m$ and $y^{m+n-1-j}\in I$ for $j\leq m-1$
(as, then, $m+n-1-j\geq n$). Thus $x+y\in r(I)$.

For $a\in A$, $(ax)^{m}=a^{m}x^{m}\in I$, whence $ax\in r(I)$.
Therefore $r(I)$ is an ideal of $A$.

\item[(ii)]Let $x\in\overline{r(I)}$ then there is $u\in r(I)$ such that $x+u=u$,
and there is $n\geq 1$ such that $u^{n}\in I$. Let us show by induction on $j\in\{0,...,n\}$ that
$u^{n-j}x^{j}\in\overline{I}$. This is clear for $j=0$. Let then
$j\in\{0,...,n-1\}$, and assume that $u^{n-j}x^{j}\in\overline{I}$ ; then
\begin{eqnarray}
u^{n-j-1}x^{j+1}+u^{n-j}x^{j} 
&=&u^{n-j-1}x^{j}(x+u) \nonumber \\
&=&u^{n-j-1}x^{j}u \nonumber \\
&=&u^{n-j}x^{j} \,\, , \nonumber 
\end{eqnarray}
whence $u^{n-j-1}x^{j+1}\in \overline{\overline{I}}=\overline{I} \,\, .$
Thus, for $j=n$, we obtain 
$$
x^{n}=u^{n-n}x^{n}\in \overline{I}\,\, ,
$$ 
whence $x\in r(\overline{I})$.

If now $I$ is saturated, then
\begin{eqnarray}
r(I)
&\subseteq& \overline{r(I)} \nonumber \\
&\subseteq& r(\overline{I}) \,\, (\text{by the above}) \nonumber \\
&=&r(I) \,\, , \nonumber
\end{eqnarray}
whence $r(I)=\overline{r(I)}$ is saturated.
\item[(iii)] That assertion is obvious.
\end{enumerate}
\end{proof}

\begin{proposition}For each saturated ideal $I$ of the $B_{1}$--algebra $A$\,\, , one has
$$
r(I)=\bigcap_{\mathcal P\in Pr_{s}(A) ; I\subseteq \mathcal P}\mathcal P\,\, .
$$
\end{proposition}
\begin{remark}For $I=\{0\}$, this is part of Theorem 5.1.
\end{remark}
\begin{proof}Let $x\in r(I)$, and let $\mathcal P\in Pr_{s}(A)$ with $I\subseteq \mathcal P$ ; then, for some $n\geq 1$
$x^{n}\in I$, whence $x^{n}\in \mathcal P$ and $x\in \mathcal P$ :
$$
r(I)\subseteq \bigcap_{\mathcal P\in Pr_{s}(A) ; I\subseteq \mathcal P}\mathcal P \,\, .
$$

Let now $y\in A$, $y\notin r(I)$, and denote by $\pi$
the canonical projection
$$
\pi : A \twoheadrightarrow A_{0}:=\displaystyle\frac{A}{\mathcal R_{I}} \,\, .
$$
As $I$ is saturated, one has
$$
\forall n\geq 1 \,\, y^{n}\notin{\overline{I}} \,\, ,
$$
whence
$$
\forall n\geq 1 y^{n}\not{\mathcal R_{I}} 0 \,\, ,
$$
or
$$
\forall n\geq 1 \,\, \pi(y)^{n}=\pi(y^{n}) \neq \overline{0} \,\, .
$$
Therefore $\pi(y)\notin Nil(A_{0})$, whence, according to Theorem 5.1,
there exists a saturated prime ideal $\mathcal P_{0}$ of $A_{0}$ such that
$\pi(y)\notin \mathcal P_{0}$. But then $\mathcal P:=\pi^{-1}(\mathcal P_{0})$
is a saturated prime ideal of $A$ containing $I$ with $y\notin \mathcal P$, whence
$$
y\notin \bigcap_{\mathcal P\in Pr_{s}(A) ; I\subseteq \mathcal P}\mathcal P\,\, .
$$
\end{proof}
\newpage

\section{Topology of spectra}
\label{sec:6}

We can now establish the basic topological properties of the spectra $Pr_{s}(A)$
(analogous, in our setting, to Corollary 1.1.8 and Proposition 1.1.10(ii) 
of \cite{6}).

\begin{theorem}$Pr_{s}(A)$ and $MaxSpec(A)$ are $T_{0}$ and
quasi--compact.
\end{theorem}
\begin{proof}According to Theorem 3.1, $Pr_{s}(A)$ and $MaxSpec(A)$ are homeomorphic, therefore it is enough to
establish the result for $Pr_{s}(A)$.

Let $\mathcal P$ and $\mathcal Q$ denote two different points of $Pr_{s}(A)$ ;
then either $\mathcal P\nsubseteq \mathcal Q$ or $\mathcal Q\nsubseteq \mathcal P$.
Let us for instance assume that $\mathcal P\nsubseteq \mathcal Q$ ; then 
$\mathcal Q\notin W(\mathcal P)$ ; set
$$
O:=Pr_{s}(A)\cap (Pr(A)\setminus W(\mathcal P))\,\, .
$$
Then $O$ is an open set in $Pr_{s}(A)$, $\mathcal Q\in O$ and, obviously,
$\mathcal P\notin O$. Therefore $Pr_{s}(A)$ is $T_{0}$.

Let $(U_{i})_{i\in I}$ denote an open cover of $Pr_{s}(A)$ :
$$
Pr_{s}(A)=\bigcup_{i\in I}U_{i}\,\, ;
$$
each $Pr_{s}(A)\setminus U_{i}$ is closed, whence $Pr_{s}(A)\setminus U_{i}=Pr_{s}(A)\cap W(S_{i})$
for some subset $S_{i}$ of $A$. Therefore $Pr_{s}(A)\cap (\bigcap_{i\in I}W(S_{i}))=\emptyset$,
 \it i.e. \rm $Pr_{s}(A)\cap W(\bigcup_{i\in I}S_{i})=\emptyset$. Therefore $Pr_{s}(A)\cap W(\overline{<\bigcup_{i\in I}S_{i}>})=\emptyset$, whence, according to 
Theorem 3.3, $\overline{<\bigcup_{i\in I}S_{i}>}=A$. Let $J=<\bigcup_{i\in I}S_{i}>$ ; then $1\in \overline{J}$, 
hence there is $x\in J$ such that $1+x=x$.
Furthermore, there exist $n\in \mathbf N$, $(i_{1},...,i_{n})\in I^{n}$ , $x_{i_{k}}\in S_{i_{k}}$ and
$(a_{1},...,a_{n})\in A^{n}$ such that $x=a_{1}x_{i_{1}}+...+a_{n}x_{i_{n}}$. But then
$$
1+a_{1}x_{i_{1}}+...+a_{n}x_{i_{n}}=a_{1}x_{i_{1}}+...+a_{n}x_{i_{n}}
$$
whence
$$
1\in \overline{<\{x_{i_{1}},...,x_{i_{n}}\}>}\subseteq \overline{\bigcup_{j=1}^{n}S_{i_{j}}}
$$
and
$$
\overline{\bigcup_{j=1}^{n}S_{i_{j}}}=A\,\, .
$$
It follows that
$$
Pr_{s}(A)\cap W(\bigcup_{j=1}^{n}S_{i_{j}})=\emptyset \,\, ,
$$
that is
$$
Pr_{s}(A)\cap \bigcap_{j=1}^{n}W(S_{i_{j}})=\emptyset \,\, ,
$$
or
$$
Pr_{s}(A)=\bigcup_{j=1}^{n}U_{i_{j}} \,\, :
$$
$Pr_{s}(A)$ is quasi--compact.
\end{proof}

For $f\in A$, let
\begin{eqnarray}
 D(f)
&:=&Pr_{s}(A)\setminus(Pr_{s}(A)\cap W(\{f\})) \nonumber \\
&=&\{\mathcal P \in Pr_{s}(A) \vert f\notin \mathcal P \}. \nonumber
\end{eqnarray}

\begin{proposition}
\begin{enumerate} 
\item Each $D(f)$($f\in A$) is open and quasi--compact in $Pr_{s}(A)$
(see \cite{6}, Proposition 1.1.10 (ii)).
\item The family $(D(f))_{f\in A}$ is an open basis for $Pr_{s}(A)$ 
(see \cite{6}, Proposition 1.1.10(i));
 in particular, the open quasi--compact sets constitute an open basis.
 \item A subset $O$ of $Pr_{s}(A)$ is open and quasi--compact 
 if and only if it is of the form $Pr_{s}(A)\cap W(I)$
 for $I$ an ideal of finite type in $A$.
 \item The family of open quasi--compact subsets of $Pr_{s}(A)$
 is stable under finite intersections.
 \item Each irreducible closed set in $Pr_{s}(A)$ has a unique generic point
 (see \cite{6}, Corollary 1.1.14(ii)).
\end{enumerate}
\end{proposition}
\begin{proof}
 \begin{enumerate}
  \item The openness of $D(f)$ is obvious.
  
Let us assume $D(f)=\bigcup_{i\in I}U_{i}$, where the $U_{i}$'s are open sets in $D(f)$. Each $U_{i}$
can be written as $$U_{i}=D(f)\cap V_{i}\,\, ,$$ for $V_{i}$ an open set in $Pr_{s}(A)$, \it i.e. \rm
$Pr_{s}(A)\setminus V_{i}=W(S_{i})$ for $S_{i}$ a subset of $A$. Then
$$
D(f)\subseteq \bigcup_{i\in I}V_{i}=Pr_{s}(A)\setminus(\bigcap_{i\in I}W(S_{i})) \,\, ,
$$
whence
$$
Pr_{s}(A)\cap W(\bigcup_{i\in I}S_{i})\subseteq W(\{f\})\,\, ,
$$
that is, setting
$$
S:=\bigcup_{i\in I}S_{i} \,\, ,
$$
$$
f\in \bigcap_{\mathcal P\in W(S)\cap Pr_{s}(A)}\mathcal P=\bigcap_{\mathcal P\in Pr_{s}(A) ; 
S\subseteq \mathcal P}\mathcal P \,\, .
$$
Therefore, by Proposition 5.5, $f\in r(\overline{<S>})$ : there is $n\geq 1$
such that $f^{n}\in \overline{<S>}$. Thus, there is $g\in <S>$ such that $f^{n}+g=g$ ;
one has $g=\sum_{j=1}^{m} a_{j}s_{j}$ for $a_{j}\in A$, $s_{j}\in S$ ; for each $j\in\{1,...,m\}$,
$s_{j}\in S_{i_{j}}$ for some $i_{j}\in I$. Let $S_{0}=\{s_{1},...,s_{m}\}$ ;
then $g\in <\bigcup_{j=1}^{n}S_{i_{j}}>$, whence $f^{n}\in\overline{<\bigcup_{j=1}^{m}S_{i_{j}}>}$,
and reading the above argument in reverse order with $S$ replaced by $\bigcup_{j=1}^{m}S_{i_{j}}$
yields that
$$
D(f)=\bigcup_{j=1}^{m}U_{i_{j}}\,\, ,
$$
whence the quasi--compactness of $D(f)$.
\item Let $U$ be an open set in $Pr_{s}(A)$,
and $\mathcal P\in U$. We have $Pr_{s}(A)\setminus U=Pr_{s}(A)\cap W(S)$ for some subset $S$ of $A$.
As $\mathcal P\notin W(S)$, $S\nsubseteq \mathcal P$, whence there is
an $s\in S$ with $s\notin \mathcal P$. It is now clear that $\mathcal P\in D(s)$ and
$$
D(s)\subseteq Pr_{s}(A)\setminus W(S)=U \,\, .
$$
\item Let $O\subseteq Pr_{s}(A)$ be open and quasi--compact ;
according to $(2)$, one may write $O=\bigcup_{j\in J}D(f_{j})$
with $f_{j}\in A$. But then, there is a finite subset $J_{0}$ of $J$
such that $O=\bigcup_{j\in J_{0}}D(f_{j})$. Now
\begin{eqnarray}
Pr_{s}(A)\setminus O \nonumber
&=&\bigcap_{j\in J_{0}} D(f_{j}) \nonumber \\
&=& Pr_{s}(A)\cap W(<f_{j}\vert j\in J_{0}>) \nonumber
\end{eqnarray}
is of the required type.

Conversely, if $Pr_{s}(A)\setminus O=Pr_{s}(A)\cap W(I)$ with $I=<g_{1},...,g_{n}>$, it is clear
that $O=\bigcup_{i=1}^{n}D(g_{i})$; as a finite union of quasi--compact subspaces of $Pr_{s}(A)$,
$O$ is therefore quasi--compact.

\item
Let $O_{1},...,O_{n}$ denote quasi--compact open subsets of $Pr_{s}(A)$ ; then, according to 
(iii), we may write
$$
Pr_{s}(A)\setminus O_{j}=Pr_{s}(A)\cap W(I_{j})
$$
for some finitely generated ideal $I_{j}$ of $A$.
Thus
\begin{eqnarray}
Pr_{s}(A)\setminus(O_{1}\cap ... \cap O_{m})
&=& \bigcup_{j=1}^{m}(Pr_{s}(A)\setminus O_{j}) \nonumber \\
&=& \bigcup_{j=1}^{m}(Pr_{s}(A)\cap W(I_{j})) \nonumber \\
&=&Pr_{s}(A)\cap \bigcup_{j=1}^{m}W(I_{j}) \nonumber \\
&=&Pr_{s}(A)\cap W(\prod_{j=1}^{m}I_{j}) \nonumber \\
&=& Pr_{s}(A)\cap W(I_{1}...I_{m}) \,\, ,\nonumber 
\end{eqnarray}
whence, according to (iii), $O_{1}\cap ... \cap O_{m}$ is
quasi--compact, as $I_{1}...I_{m}$ is finitely generated.
\item 
Let $F$ denote an irreducible closed set in $Pr_{s}(A)$ ; then $F=Pr_{s}(A)\cap W(S)$
for $S$ a subset of $A$. We have  seen above that, setting $I:=\overline{<S>}$, one has
$F=Pr_{s}(A)\cap W(I)$. As $F$ is not empty, $I\neq A$. Let us assume $ab\in I$ ;
then, for each $\mathcal P\in F$, one has $ab\in I\subseteq P$, whence $a\in \mathcal P$
or $b\in \mathcal P$, \it i.e. \rm $\mathcal P\in F\cap W(\{a\})$ or $\mathcal P\in F\cap W(\{b\})$ :
$$
F=(F\cap W(\{a\}))\cup (F\cap W(\{b\})) \,\, .
$$ 
As $F$ is irreducible, it follows that either $F=F\cap W(\{a\})$ or 
$F=F\cap W(\{b\})$.
In the first case we get $F\subseteq W(\{a\})$, \it i.e. \rm
$$
a\in \bigcap_{\mathcal P\in Pr_{s}(A) ; I\subseteq \mathcal P}\mathcal P=I
(\text{Proposition 5.5})\, ;
$$
similarly, in the second case, $b\in I$ : $I$ is prime.
But then
\begin{eqnarray}
\overline{\{I\}} \nonumber
&=&Pr_{s}(A)\cap W(I) \nonumber \\
&=&F \nonumber
\end{eqnarray}
and $I$ is a generic point for $F$.

It is unique as, in a $T_{0}$--space, an (irreducible) closed set admits \bf at most one \rm generic point
(see \cite{6}, (0.2.1.3)).
\end{enumerate}
\end{proof}

\begin{corollary}$Pr_{s}(A)$ and $MaxSpec(A)$ are spectral spaces
in the sense of Hochster (\cite{7},p. 43).
\end{corollary}
\begin{theorem}(cf. \cite{6}, Corollary 1.1.14) Let $F=Pr_{s}(A)\cap W(S)$ be a nonempty closed set in $Pr_{s}(A)$ ;
then $F$ is homeomorphic to $Pr_{s}(B)$, where $B:=\displaystyle\frac{A}{\mathcal R_{I}}$
with $I:=\overline{<S>}$.
\end{theorem}
\begin{proof}
As seen above, one has $F=Pr_{s}(A)\cap W(I)$, whence, as $F\neq \emptyset$, $I\neq A$.
Let $A_{0}:=\displaystyle\frac{A}{\mathcal R_{I}}$, and let $\pi : A \rightarrow A_{0}$
denote the canonical projection.

Let us now define
\begin{eqnarray}\psi
&:& Pr_{s}(A_{0}) \rightarrow F \nonumber \\
&& \mathcal Q \mapsto \pi^{-1}(\mathcal Q) \,\, .\nonumber
\end{eqnarray}

Then $\psi$ is well--defined (as $\pi^{-1}(\mathcal Q)$ is a saturated prime ideal of $A$ that
contains $I$), and injective (as, for each $\mathcal Q\in Pr_{s}(A_{0})$,
$\pi(\psi(\mathcal Q))=Q$). 

Let $\mathcal P\in F$ ; then $\pi(\mathcal P)$ is an ideal of $A_{0}$.
Let us assume $\pi(v)\in \overline{\pi(\mathcal P)}$ ; then
$$
\pi(v)+\pi(a)=\pi(a)
$$
for some $a\in \mathcal P$,
that is
$$
\pi(a+v)=\pi(a)\,\, .
$$
But then
$$
a+v+i=a+i
$$
for some $i\in I$, whence
$$
v+(a+i)=a+i
$$
As $a+i\in \mathcal P$ and $\mathcal P$ is saturated,
it follows that $v\in \mathcal P$ : $\pi(\mathcal P)$ is saturated.

Furthermore , if $\pi(1)\in \pi(\mathcal P)$, one has $\pi(1)+\pi(v)=\pi(v)$ for some
$v\in \mathcal P$, whence there is $w\in I$ such that $1+v+w=v+w$,
whence $1+v+w\in \mathcal P$ and (as $\mathcal P$ is saturated) $1\in \mathcal P$
and $\mathcal P=A$, a contradiction. Therefore $\pi(\mathcal P)\neq A_{0}$.

Let us assume $\pi(x)\pi(y)\in \pi(\mathcal P)$ :
then $xy+i=q+i$ for some $i\in I$, whence
$$
(x+i)(y+i)=xy+xi+iy+i^{2}\in \mathcal P \,\, ,
$$
and $x+i\in \mathcal P$ or $y+i\in\mathcal P$ ; as $\mathcal P$ is saturated,
it follows that $x\in \mathcal P$ or $y\in \mathcal P$,
whence $\pi(x)\in \pi(\mathcal P)$ or $\pi(y)\in \pi(\mathcal P)$ : $\pi(\mathcal P)$ is prime.

As $\mathcal P$ is saturated, one sees in the same way that $\psi(\pi(\mathcal P))=\pi^{-1}(\pi(\mathcal P))=\mathcal P$, whence
$\psi$ is surjective.

Let $G:=F\cap W(S_{0})$ be closed in $F$ ;
then $\mathcal P\in \psi^{-1}(G)$ if and only if $\psi(\mathcal P)\in F \cap W(S_{0})$, that is
$S\subseteq \pi^{-1}(\mathcal P)$ and $S_{0}\subseteq \pi^{-1}(\mathcal P)$, \it i.e. \rm $\pi(S\cup S_{0})\subseteq \mathcal P$ :
$$
\psi^{-1}(G)=Pr_{s}(A_{0})\cap W(\pi(S\cup S_{0}))
$$
is closed in $F$, and $\psi$ is continuous.

Let now $H:=Pr_{s}(A_{0})\cap W(\bar{G})$ be closed in $Pr_{s}(A_{0})$, and let $\mathcal Q \in Pr_{s}(A_{0})$ ;
as $\pi$ is surjective, $\bar{G}\subseteq \mathcal Q$ if and only if
$\pi^{-1}(\bar{G})\subseteq \pi^{-1}(\mathcal Q)=\psi(\mathcal Q)$, and it follows that
$$
\psi(H)=F\cap W(\pi^{-1}(\bar{G}))
$$
is closed in $F$.
Therefore $\psi$ is an homeomorphism.
\end{proof}

\newpage

\section{Remarks on the one--generator case}
\label{sec:7}

Let us now consider the case of a nontrivial monogenic $B_{1}$--algebra containing strictly $B_{1}$, \it i.e. \rm $A=\displaystyle\frac{B_{1}[x]}{\sim}$
is a quotient of the free algebra $B_{1}[x]$ with $x\nsim 0$, $x\nsim 1$. Denote by $\alpha$ the image of $x$ in $A$ ;
then $\alpha\notin\{0,1\}$, and $\alpha$ generates $A$ as a $B_{1}$--algebra.

Let us suppose that, for some $(u,v)\in A^{2}$, $\alpha u=1+\alpha v$ ; then $\alpha$
is not nilpotent, as from $\alpha^{n}=0$ would follow 
$$
0=\alpha^{n}v=\alpha^{n-1}(\alpha v)
=\alpha^{n-1}(1+\alpha u)=\alpha^{n-1}+\alpha^{n}u=\alpha^{n-1}\,\, ,
$$
whence $\alpha^{n-1}=0$ and, by induction on $n$, $1=\alpha^{0}=0$, a contradiction.

Therefore three cases may appear
\begin{enumerate}
\item[(i)]$\alpha$ is nilpotent.
\item[(ii)]$\alpha$ is not nilpotent and there does not exist $(u,v)\in A^{2}$
such that $\alpha u=1+\alpha v$.
\item[(iii)]($\alpha$ is not nilpotent) and there exists $(u,v)\in A^{2}$
such that $\alpha u=1+\alpha v$.
\end{enumerate}

In case (i), any prime ideal of $A$ must contain $\alpha$, hence contain $\alpha A$;
the ideal $\alpha A$ is, according to the above remark, saturated, and is not contained in a 
strictly bigger saturated ideal other than $A$ itself (in both cases, as any element of $A$ not
in $\alpha A$ is of the shape $1+\alpha x$). Therefore $Pr_{s}(A)=\{\alpha A\}$, whence
$Nil(A)=\alpha A$. In this case we see that
$$
\displaystyle\frac{A}{\mathcal R_{Nil(A)}}\simeq B_{1}\,\, .
$$

In cases (ii) and (iii), no power of $\alpha$ belongs to $Nil(A)$ ; as
$Nil(A)$ is saturated, it follows that $Nil(A)=\{0\}$. In fact, $A$ is integral,
whence $\{0\}\in Pr_{s}(A)$.
If $\mathcal P\in Pr_{s}(A)$ and $\mathcal P\neq \{0\}$, then $\mathcal P$ contains
some power of $\alpha$, hence contains $\alpha$, hence contains $\alpha A$.
As above we see that $\mathcal P=\alpha A$ ; but, in case (iii), $\alpha A$
is not saturated. In case (ii) it is easy to see that $\alpha A$ is prime and saturated.
Therefore 
\begin{enumerate}
\item In case (ii), $Pr_{s}(A)=\{\{0\},\alpha A\}$ ; $\{0\}$ is a generic point,
that is
$$
\overline{\{\{0\}\}}=Pr_{s}(A)\,\, ,
$$
 and $\alpha A$ a \lq\lq closed point\rq\rq
($\{\alpha A\}$ is closed) ;
\item In case (iii), $Pr_{s}(A)=\{\{0\}\}$.
\end{enumerate}

One may remark that $B_{1}[x]$ itself falls into case (ii).

In \cite{9}, pp. 75--79, we have enumerated (up to isomorphism) monogenic 
$B_{1}$--algebras of cardinality $\leq 5$. It is easy to see 
where these algebras fall in the above classification ; we keep the numbering used
in \cite{9}. Let then $3\leq \vert A \vert \leq 5$.
We have the following repartition

Case $(i)$ : (6),(8),(12),(15),(18),(24)

Case $(ii)$: (7),(10),(11),(16),(19),(25),(26)

Case $(iii)$: (5),(9),(13),(14),(17),(20),(21),(22),(23),(27),(28)

\newpage
\section{Acknowledgement}
This paper was written during a stay at I.H.E.S. (December 2010-March 2011) ; a preliminary version appeared as
prepublication IHES/M/11/06.
I am grateful both to the staff and to the colleagues who managed to make this stay pleasant and stimulating.
I am also deeply indebted to Profesor Alain Connes for his constant moral support.

I am very grateful for the invitations to lecture upon this work at the JAMI conference \it Noncommutative geometry
and arithmetic \rm(Johns Hopkins University, March 22-25 2011), and at the \it AGATA Seminar \rm (Montpellier, April 7, 2011) ; for these, I thank respectively Professors Alain Connes and Caterina Consani, and Professor Vladimir Vershinin.

\bibliographystyle{elsarticle-num}
\bibliography{<your-bib-database>}

\end{document}